\documentclass{article}%
\usepackage{color}
\usepackage{graphicx}
\usepackage[intlimits]{amsmath}
\usepackage{latexsym}
\usepackage{amsfonts}
\usepackage{amssymb}%
\makeatletter
\newfont{\footsc}{cmcsc10 at 8truept}
\newfont{\footbf}{cmbx10 at 8truept}
\newfont{\footrm}{cmr10 at 10truept}
\newtheorem{theorem}{Theorem}

\newtheorem{corollary}[theorem]{Corollary}

\newtheorem{definition}[theorem]{Definition}

\newtheorem{lemma}[theorem]{Lemma}

\newtheorem{question}[theorem]{Question}

\newenvironment{proof}[1][Proof]{\noindent{\textbf {#1}  }}  {\hfill$\Box$\bigskip}

\begin{document}

\title{Some properties and applications of odd-colorable $r$-hypergraphs \thanks{This work was supported by the
Hong Kong Research Grant Council (Grant Nos. PolyU 501212, 501913, 15302114
and 15300715) and NSF of China (Grant Nos. 11231004, 11571123 and 11101263)
and by a grant of \textquotedblleft The First-class Discipline of Universities
in Shanghai\textquotedblright.\textit{ }}}
\author{Xiying Yuan\thanks{Department of Mathematics, Shanghai University, Shanghai
200444, China; \textit{E-mail address: xiyingyuan2007@hotmail.com }} \ Liqun
Qi\thanks{Department of Applied Mathematics, The Hong Kong Polytechnic
University, HungHom, Kowloon, HongKong, \textit{Email address:
liqun.qi@polyu.edu.hk}} \ Jiayu Shao\thanks{Corresponding author}
\thanks{Department of Mathematics, Tongji University, Shanghai, China,
\textit{Email address: jyshao@tongji.edu.cn}} \ Chen Ouyang\thanks{Department
of Applied Mathematics, The Hong Kong Polytechnic University, HungHom,
Kowloon, HongKong, \textit{Email address: oychen26@126.com}}}
\maketitle

\begin{abstract}
Let $r\geq2$ and $r$ be even. An $r$-hypergraph $G$ on $n$ vertices is called
odd-colorable if there exists a map $\varphi:[n]\rightarrow\lbrack r]$ such
that for any edge $\{j_{1},j_{2},\cdots,j_{r}\}$ of $G$, we have
$\varphi(j_{1})+\varphi(j_{2})+\cdot\cdot\cdot+\varphi(j_{r})\equiv
r/2(\operatorname{mod}r).$  In this paper, we first determine that, if $r=2^{q}(2t+1)$ and $n\ge 2^{q}(2^{q}-1)r$, then the maximum chromatic number in the class of the odd-colorable $r$-hypergraphs on $n$ vertices is $2^q$,  which answers a question raised
by V. Nikiforov recently in [V. Nikiforov, Hypergraphs and hypermatrices with
symmetric spectrum. Prinprint available in arXiv:1605.00709v2, 10 May, 2016]. We also study some applications of the symmetric spectral property of the odd-colorable $r$-graphs given in that same paper by V. Nikiforov. We show that the Laplacian spectrum
and the signless Laplacian spectrum of an $r$-hypergraph $G$ are equal if and only if $G$ is odd-colorable, and then study some further applications of these spectral properties.



\end{abstract}

\section{Introduction}

Denote the set $\{1,2,\cdot\cdot\cdot,n\}$ by $[n]$. An $r$-hypergraph $G$
$=(V(G),E(G))$ on $n$ vertices is an $r$-uniform hypergraph each of whose edges contains exactly $r$ vertices (\cite{Berge1973}). In this paper, $r$-hypergraph is simply called $r$-graph for convenience.
A 2-graph is just an ordinary graph.

The definition of odd-coloring for tensors (it is called $r$-matrices in
\cite{Nikiforov2016}) was introduced in \cite{Nikiforov2016}, we just focus on
its version for $r$-graph as follows.

\begin{definition}
Let $r\geq2$ and $r$ be even. An $r$-graph $G$ with $V(G)=[n]$ is called
odd-colorable if there exists a map $\varphi:[n]\rightarrow\lbrack r]$ such
that for any edge $\{j_{1},j_{2},\cdots,j_{r}\}$ of $G$, we have
\[
\varphi(j_{1})+\cdot\cdot\cdot+\varphi(j_{r})\equiv r/2(\operatorname{mod}r).
\]
The function $\varphi$ is called an odd-coloring of $G.$
\end{definition}


The following concept of odd-bipartite $r$-graphs was taken from \cite{HuQi2014}, and this concept acts as generalizations of the ordinary bipartite graphs.
\begin{definition} \cite{HuQi2014}
An $r$-graph $G = (V,E)$ is called odd-bipartite, if $r$ is even and there
exists some proper subset $V_{1}$ of $V$ such that each edge of $G$ contains exactly odd number of
vertices in $V_{1}$.
\end{definition}

The odd-bipartite $r$-graphs  were also called odd-transversal $r$-graphs in literature (see
\cite{Berge1973}, \cite{Cowen2007}, or \cite{Nikiforov2016}). The connection between odd-bipartiteness and spectra of
$r$-graphs  was studied in \cite{HuQi2014}, \cite{HuQiXie2015}, \cite{Nikiforov2014} and \cite{ShaoShanWu}.

In \cite{Nikiforov2016}, it was proved that an odd-bipartite graph is always odd-colorable (see
Proposition 11 in \cite{Nikiforov2016}), and furthermore, in the case $r\equiv2($mod$4)$, then
$G$ is odd-colorable if and only if $G$ is odd-bipartite (see Proposition 12
in \cite{Nikiforov2016}). 

An $r$-graph $G$ is called $k$-chromatic if its vertices can be partitioned into
$k$ sets so that each edge intersects at least two sets. The chromatic number
$\chi(G)$ of $G$ is the smallest $k$ for which $G$ is $k$-chromatic. The
chromatic number of an odd-colorable $r$-graph is also considered in
\cite{Nikiforov2016}. Clearly, each nontrivial odd-bipartite graph has
chromatic number 2. A family of 3-chromatic odd-colorable $4k$-graphs on $n$
vertices is constructed in \cite{Nikiforov2016}. Notice that odd-colorable
$r$-graphs are defined only for even $r$. For further information about the
chromatic number of odd-colorable graph, the following question is raised in
\cite{Nikiforov2016}.

\begin{question}
\label{Chromatic number} Let $r\equiv0(\operatorname{mod}4).$ What is the
maximum chromatic number of an odd-colorable $r$-graph on $n$ vertices?
\end{question}
In section 2, we will determine that, if $r$ is even, $r=2^{q}(2t+1)$ for some integers $q,t$ and $n\ge 2^{q}(2^{q}-1)r$, then the maximum chromatic number in the class of the odd-colorable $r$-graphs on $n$ vertices is $2^q$. This result provides an answer to the above Question \ref{Chromatic number}.

\begin{definition}
\cite{HuQiXie2015} \cite{Qi2014} Let $G=(V(G),E(G))$ be an $r$-graph on $n$
vertices. The adjacency tensor of $G$ is defined as the order $r$ dimension
$n$ tensor $\mathcal{A}(G)$ whose $(j_{1}\cdots j_{r})$-entry is:
\[
(\mathcal{A}(G))_{j_{1}j_{2}\cdots j_{r}}=%
\begin{cases}
\frac{1}{(r-1)!} & \text{if $\{j_{1},j_{2},\cdots,j_{r}\}\in E(G),$}\\
0 & \text{otherwise}.
\end{cases}
\]
Let $\mathcal{D}(G)$ be an order $r$ dimension $n$ diagonal tensor, with its
diagonal entry $\mathcal{D}_{jj\cdots j}$ being the degree of vertex $j$, for all
$j\in\lbrack n]$. Then $\mathcal{L(}G\mathcal{)}=\mathcal{D(}G\mathcal{)}%
-\mathcal{A(}G\mathcal{)}$ is called the Laplacian tensor of $r$-graph $G$, and $\mathcal{Q(}G\mathcal{)}=\mathcal{D(}G\mathcal{)}%
+\mathcal{A(}G\mathcal{)}$ is called the signless Laplacian tensor of $G,$.
\end{definition}

The following general product of tensors, was defined in \cite{Shao},
which is a generalization of the matrix case. Let $\mathcal{A}$ and $\mathcal{B}$ be dimension $n$ and order $m\geq2$
and $k\geq1$ tensors, respectively. The product $\mathcal{AB}$ is the
following tensor $\mathcal{C}$ of dimension $n$ and order $(m-1)(k-1)+1$ with entries:
\begin{equation}
\mathcal{C}_{i\alpha_{1}\cdots\alpha_{m-1}}=\sum_{i_{2},\cdots,i_{m}\in\lbrack
n_{2}]}\mathcal{A}_{ii_{2}\cdots i_{m}}\mathcal{B}_{i_{2}\alpha_{1}}%
\cdots\mathcal{B}_{i_{m}\alpha_{m-1}}, \label{1}%
\end{equation}
where $i\in\lbrack n],\alpha_{1},\cdots,\alpha_{m-1}\in\lbrack n]^{k-1}$.

Let $\mathcal{T}$ be an order $r$ dimension $n$ tensor, let $x=(x_{1}%
,\cdot\cdot\cdot,x_{n})^{T}\in\mathbb{C}^{n}$ be a column vector of dimension
$n$. Then by (1) $\mathcal{T}x$ is a vector in $\mathbb{C}^{n}$ whose $j$th
component is as the following%

\begin{equation}
(\mathcal{T}x)_{j}=\sum_{j_{2},\cdots,j_{r}=1}^{n}\mathcal{T}_{jj_{2}\cdots
j_{r}}x_{j_{2}}\cdots x_{j_{r}}. \label{2}%
\end{equation}

Let $x^{[r]}=(x_{1}^{r},\cdots,x_{n}^{r})^{T}$. Then (see \cite{ChangPZ2008}
\cite{Qi2014}) a number $\lambda\in\mathbb{C}$ is called an eigenvalue of the
tensor $\mathcal{T}$ of order $r$ if there exists a nonzero vector
$x\in\mathbb{C}^{n}$ satisfying the following eigenequations%

\begin{equation}
\mathcal{T}x=\lambda x^{[r-1]}, \label{3}%
\end{equation}
and in this case, $x$ is called an eigenvector of $\mathcal{T}$ corresponding
to eigenvalue $\lambda$.  The spectral radius of $\mathcal{T}$ is defined as
\[
\rho(\mathcal{T})=max\{|\mu|:\mu\text{ is an eigenvalue of }\mathcal{T}\}.
\]

In order to define the spectra of tensors, we first need to define the determinants of tensors. Originally the determinants of tensors were defined as the resultants of some corresponding system of homogeneous equations on $n$ variables. Here we give the following equivalent definition of the determinants of tensors.

\begin{definition}
Let $\mathcal{A}$ be an order $m$ dimension $n$ tensor with $m\ge 2$. Then its determinant $det(\mathcal{A})$ is defined to be the unique polynomial on the entries of $\mathcal{A}$ satisfying the following three conditions:
\vskip 0.1cm

\noindent (1) $det(\mathcal{A})=0$ if and only if the system of homogeneous equations $\mathcal{A}x=0$ has a nonzero solution.

\vskip 0.1cm

\noindent (2) $det(\mathcal {A})=1$, when $\mathcal {A}=\mathbb {I}$ is the unit tensor.

\vskip 0.1cm

\noindent (3) $det(\mathcal{A})$ is an irreducible polynomial on the entries of $\mathcal{A}$, when the entries of $\mathcal{A}$ are viewed as distinct independent variables.
\end{definition}

\begin{definition}
\label{spectrum} Let $\mathcal{A}$ be an order $m\ge 2$ dimension $n$ tensor. Then the characteristic polynomial of $\mathcal {A}$ is defined to be the determinant $det(\lambda \mathcal{I}- \mathcal{A})$. The (multi)-set of roots of the characteristic polynomial of $\mathcal{A}$ (counting multiplicities) is called the spectrum of $\mathcal{A}$, denoted by $Spec(\mathcal{A})$.
\end{definition}

If the tensor $\mathcal{T}$ and $-\mathcal{T}$ have the same spectrum (i.e.,
the spectrum of $\mathcal{T}$ is symmetric about the origin), then the spectrum
of $\mathcal{T}$ is said to be symmetric in this paper. In \cite{Nikiforov2016}, Nikiforov studied some symmetric spectral property of the odd-colorable $r$-graphs. He proved that for an $r$-graph $G$, $Spec(\mathcal{A}(G))=-Spec(\mathcal{A}(G))$ if and only if $r$ is even and $G$ is odd-colorable. This result  solves a problem in \cite{Pearson2014} about $r$-graphs with symmetric spectrum and disproves a conjecture in \cite{Zhou2014}.

In Section 3, we will give some applications and consequences of these symmetric spectral property of the odd-colorable $r$-graphs given in \cite{Nikiforov2016}. In particular, we obtain (in Theorem \ref{main result of section 2}) some further symmetric spectral property of the odd-colorable $r$-graphs related to the Laplacian and signless Laplacian spectrum of an $r$-graph $G$. The proof of the disconnected case of this result need to use the Perron-Frobenius Theorem on nonnegative weakly irreducible tensors, the relation between the (Laplacian and signless Laplacian) spectra of an disconnected $r$-graph $G$ with that of all the connected components of $G$, and so on. We also use these results to study the Question \ref{Hspectra} proposed in \cite{ShaoShanWu} about the relations between H-spectra of $\mathcal{L}(G)$ and $\mathcal{Q}(G)$ with the spectra of $\mathcal{L}(G)$ and $\mathcal{Q}(G)$, and obtain an affirmative answer to Question \ref{Hspectra} for the remaining unsolved case $r\equiv 2 \ (mod \ 4)$ in Theorem \ref{characterization of the inverse}.

\section{ Maximum chromatic number of an odd-colorable $r$-graph}

Let $r$ be even, then there uniquely exist two integers $q,t$ such that
$r=2^{q}(2t+1).$  In this section, we will determine that, if $n\ge 2^{q}(2^{q}-1)r$, then the maximum chromatic number in the class of the odd-colorable $r$-graphs on $n$ vertices is $2^q$. This result also provides an answer to Question \ref{Chromatic number} in \S 1.

First we prove the following upper bound on the chromatic number of the odd-colorable
$r$-graphs.

\begin{theorem}
\label{bound for chromatic number} \ Let $q\geq1,t\geq0$ be two integers and
$r=2^{q}(2t+1)$, and $G$ be an odd-colorable $r$-graph. Then $\chi(G)\leq
2^{q}.$
\end{theorem}

\begin{proof}
Suppose $|V(G)|=n$ and let the function $\varphi:[n]\rightarrow\lbrack r]$ be
an odd-coloring of $G.$ For $0\leq i\leq2^{q}-1,$ set
\[
V_{i}:=\{j:j\in\lbrack n],\text{ }\varphi(j)\equiv i(\operatorname{mod}%
2^{q})\},
\]
and some $V_{i}$ may be empty. Then the vertices set $V(G)$ can be
partitioned as%

\[
V(G)=V_{0}\cup\cdot\cdot\cdot\cup V_{2^{q}-1}.
\]
We claim that each set $V_{i}$ contains no edge of $G.$ Suppose not, let $e=\{j_{1},j_{2},\cdots,j_{r}\}$ be an edge in some $V_{i}.$
Since $j_{l}$ is in $V_{i},$ we have $\varphi(j_{l})=c_{l}2^{q}+i,$ where
$c_{l}$ is a nonnegative integer $(l=1,\cdots,r)$. Then we have
\begin{align*}
&  \varphi(j_{1})+\cdot\cdot\cdot+\varphi(j_{r})\\
&  =(c_{1}2^{q}+i)+\cdot\cdot\cdot+(c_{r}2^{q}+i)\\
&  =2^{q}(%
{\displaystyle\sum\limits_{l=1}^{r}}
c_{l})+ir\\
&  =2^{q}(%
{\displaystyle\sum\limits_{l=1}^{r}}
c_{l})+i2^{q}(2t+1).
\end{align*}

On the other hand, by the definition of odd-coloring, there exists some
integer $c$ such that
\begin{align*}
&  \varphi(j_{1})+\cdot\cdot\cdot+\varphi(j_{r})\\
&  =cr+r/2\\
&  =c2^{q}(2t+1)+2^{q-1}(2t+1).
\end{align*}
Hence we have%
\[
2^{q}(%
{\displaystyle\sum\limits_{l=1}^{r}}
c_{l})+i2^{q}(2t+1)=c2^{q}(2t+1)+2^{q-1}(2t+1),
\]
which implies that
\[
2[%
{\displaystyle\sum\limits_{l=1}^{r}}
c_{l}+(i-c)(2t+1)]=2t+1,
\]
and this is a contradiction, since it has different parity of two sides. So
each set $V_{i}$ does not span any edge and so $\chi(G)\leq2^{q}.$
\end{proof}

Now we construct a family of odd-colorable $r$-graphs on $n$ vertices to show that the upper bound given in Theorem \ref{bound for chromatic number} is sharp for all $n\geq2^{q}(2^{q}-1)r$.

\begin{lemma}
\label{bound is tight} Let $q\geq1,t\geq0$ be two integers and $r=2^{q}%
(2t+1).$ If $n\geq2^{q}(2^{q}-1)r$, then there exists a family of $2^{q}%
$-chromatic odd-colorable $r$-graphs on $n$ vertices.
\end{lemma}

\begin{proof}
For any $1\leq i<j\leq2^{q},$ we may write%
\[
j-i=2^{p_{i,j}}(2a_{i,j}+1),\quad \mbox {and} \quad  b_{i,j}=2^{q-p_{i,j}-1}(2t+1),
\]
where $0\leq p_{i,j}<q,$ $a_{i,j}\geq0$ are integers. By definition of $b_{i,j}$, we may see that $$b_{i,j}=\frac{2^{q}(2t+1)}{2^{p_{i,j}+1}}=\frac{r}{2^{p_{i,j}+1}}\leq \frac{r}{2}.$$

Now we start to construct the desired $r$-graph $G$. First we take $V(G)=[n]$. In order to define the edge set $E(G)$, we first take any subsets $V_1,\cdots, V_{2^{q}}$ of $V$ such that $V_{i}'s$ are pairwisely disjoint, and  $[n]=V_{1}\bigcup V_{2}\bigcup\cdot\cdot\cdot\bigcup V_{2^{q}}$, and $|V_{i}|\geq r(2^{q}-1)$, $1\leq i\leq 2^{q}.$

Using these subsets $V_1,\cdots, V_{2^{q}}$, we can define the following $2^{q-1}(2^{q}-1)$ families $E_{i,j}$ of $r$-subset of $[n]$. For any
$1\leq i<j\leq2^{q},$  set%
\[
E_{i,j}:=\{e:e\subset\lbrack n],|e\cap V_{i}|=r-b_{i,j}, \text{ and
}|e\cap V_{j}|=b_{i,j} \}.
\]
Finally, we define the desired $r$-graph $G$ as $V(G):=[n]$ and
$$E(G):=E_{1,2}\cup
E_{1,3}\cup\cdot\cdot\cdot\cup E_{(2^{q}-1),2^{q}}=\bigcup_{1\leq i<j\leq 2^{q}}E_{i,j}.$$
From the definition we can see that, if there exists a vertex subset $C\subseteq V$ with $|V_{i}\cap C|\ge r$ and $|V_{j}\cap C|\ge r$ for some $1\leq i<j\leq 2^{q}$, then $C$ contains some edge from the set $E_{i,j}$.

First we will show that $G$ is odd-colorable. Define the map $\varphi
:[n]\rightarrow\lbrack r]$ by letting
\[
\varphi(v)=i,\text{ if }v\in V_{i}.
\]
We shall check that the function $\varphi$ is an odd-coloring of $G$. Let
$e\in E(G).$ If $e\in E_{i,j}$ for some $1\leq i<j\leq 2^{q}$, then
\begin{align*}%
{\displaystyle\sum\limits_{l\in e}}
\varphi(l)  &  =(r-b_{i,j})i+b_{i,j}j\\
&  =ri+b_{i,j}(j-i)\\
&  =ri+2^{q-1}(2t+1)(2a_{i,j}+1)\\
&  =ri+2^{q}a_{i,j}(2t+1)+2^{q-1}(2t+1)\\
&  =ri+ra_{i,j}+r/2\\
&  \equiv r/2(\operatorname{mod}r).
\end{align*}
Hence if $\{j_{1},j_{2},\cdots,j_{r}\}$ is an edge of $G,$ then $\varphi
(j_{1})+\varphi(j_{2})+\cdot\cdot\cdot+\varphi(j_{r})\equiv
r/2(\operatorname{mod}r).$ So $G$ is odd-colorable.

Now we will show $\chi(G)=2^{q}.$  First from Theorem \ref{bound for chromatic number}  we know that $\chi(G)\leq2^{q},$ since $G$ is odd-colorable. Next we show that $\chi(G)\geq2^{q}$. Suppose not,  assume that we have a partition $V(G)=C_{1}\cup\cdot\cdot\cdot\cup
C_{2^{q}-1}$ such that $C_{i}^{\prime}s$ are pairwisely disjoint and there is
no edge in each $C_{i}$ for $1\leq i\leq2^{q}-1,$  we will use the pigeonhole principle to get a contradiction.

\vskip 0.38cm

We first define an auxiliary matrix $T=(t_{ij})$ with $2^q$ rows and $2^q-1$ columns such that
$$t_{ij}=|V_{i}\cap C_{j}| \qquad  (i=1,\cdots, 2^q; \ j=1,\cdots, 2^q-1)$$
Then the $i$-th row sum of the matrix $T$ is
$$\sum_{j=1}^{2^{q}-1}t_{ij}=\sum_{j=1}^{2^{q}-1}|V_{i}\cap C_{j}|=|V_{i}|\geq r(2^{q}-1)$$
This implies that for each $i=1,\cdots, 2^q$, there exists some $j\in \{1,\cdots, 2^q-1\}$ such that $t_{ij}\ge r$.

\vskip 0.18cm

An entry $t_{ij}$ of the matrix $T$ is called good if $t_{ij}\ge r$. Then the above arguments shows that every row of $T$ contains at least one good entry, so altogether $T$ contains at least $2^q$ good entries since $T$ has $2^q$ rows. On the other hand, $T$ has $2^q-1$ columns. So by the pigeonhole principle, there exists some column of $T$ containing at least two good entries, say $t_{ik}\ge r$ and $t_{jk}\ge r$ are good entries $(i<j)$. This implies that the class $C_{k}$ contains an edge from the set $E_{i,j}$, a contradiction.
\end{proof}

Combining Theorem \ref{bound for chromatic number} and Lemma \ref{bound is tight}, we can obtain the following:

\begin{theorem}
\label{max chromatic number} If $r=2^{q}%
(2t+1)$ is even and $n\ge 2^{q}(2^{q}-1)r$, then the maximum chromatic number in the class of the odd-colorable $r$-graphs on $n$ vertices is $2^q$.
\end{theorem}

Obviously, the special case $q\ge 2$ of Theorem \ref{max chromatic number}  also provides an answer to Question \ref{Chromatic number} in \S 1.

\section{Some applications of the symmetric spectral property of the odd-colorable $r$-graphs}

In \cite{Nikiforov2016}, Nikiforov studied some symmetric spectral property of the odd-colorable $r$-graphs. He  proved the following result of the odd-colorable $r$-graphs.

\begin{theorem}
\cite{Nikiforov2016}  \label{Nikiforov}
 Let $G$ be an  $r$-graph. Then $Spec(\mathcal{A}(G))=-Spec(\mathcal{A}(G))$ if and only if $r$ is even and $G$ is odd-colorable.
\end{theorem}

In this section, we will give some applications and consequences of this symmetric spectral property of the odd-colorable $r$-graphs. In particular, we obtain some further symmetric spectral property of the odd-colorable $r$-graphs related to the Laplacian and signless Laplacian spectrum of an $r$-graph $G$ (see Theorem \ref{main result of section 2} below). We also use these results to study the Question \ref{Hspectra} proposed in \cite{ShaoShanWu}  about the relations between H-spectra of $\mathcal{L}(G)$ and $\mathcal{Q}(G)$ with the spectra of $\mathcal{L}(G)$ and $\mathcal{Q}(G)$, and obtain an affirmative answer to Question \ref{Hspectra} for the remaining unsolved case $r\equiv 2 \ (mod \ 4)$ in Theorem \ref{characterization of the inverse}.

\vskip 0.18cm

Recall that in Theorem 3.2 of \cite{FanKhanTan2016}, Fan et al. proved that in the case of the non-odd-bipartite connected $r$-graphs, then the following Lemma \ref{Shao and Fan} holds. Combining this with the Theorems 2.2 and 2.3 in \cite{ShaoShanWu} for the odd-bipartite  connected case, we have the following result.

\begin{lemma}
\label{Shao and Fan}
Let $G$ be a connected $r$-graph. Then $Spec(\mathcal{A}(G))=-Spec(\mathcal{A}(G))$ if and only if
$ Spec(\mathcal{L}(G))=Spec(\mathcal{Q}(G)).$
\end{lemma}

\begin{proof}
If $G$ is not odd-bipartite, the result follows from Theorem 3.2 of \cite{FanKhanTan2016}. If $G$ is odd-bipartite, the result follows from Theorems 2.2 and 2.3 of \cite{ShaoShanWu}.
\end{proof}

Combining Theorem \ref{Nikiforov} and Lemma \ref{Shao and Fan}, we can obtain that, for a connected $r$-graph $G$, its Laplacian spectrum and signless Laplacian spectrum are equal if and only if $r$ is even and $G$ is odd-colorable. In order to extend this result to the disconnected case, we need the following result which is a consequence of Corollary 4.2 of \cite{2013SSZ}.

\begin{lemma}\cite{2013SSZ}\label{spectrum of union}
Let $G$ be an $r$-graph of order $n$, $G_{1},G_{2},\cdot\cdot\cdot,G_{k}$ be all the connected components of $G$, with orders $n_1,\cdots, n_k$, respectively. Then
$$Spec(\mathcal{A}(G))=\bigcup_{i=1}^k Spec(\mathcal{A}(G_i))^{(r-1)^{n-n_i}},$$
$$Spec(\mathcal{L}(G))=\bigcup_{i=1}^k Spec(\mathcal{L}(G_i))^{(r-1)^{n-n_i}},$$
$$Spec(\mathcal{Q}(G))=\bigcup_{i=1}^k Spec(\mathcal{Q}(G_i))^{(r-1)^{n-n_i}},$$
where the notation $S^t$ means the repetition of $t$ times of the multi-set $S$.
\end{lemma}

We also need some more preliminaries for the study of disconnected case of the symmetric spectral property of the Laplacian and signless Laplacian spectrum of odd-colorable $r$-graphs.

\begin{lemma}\cite{ShaoShanWu} \label{radius to spectrum}
Let $G$ be a connected $r$-graph. Then $\rho(\mathcal{L}(G))=\rho(\mathcal{Q}(G))$ if and only if $Spec(\mathcal{L}(G))=Spec(\mathcal{Q}(G))$.
\end{lemma}

In \cite{FGH2013}, the weak irreducibility of nonnegative tensors was defined as follows.

\begin{definition} \cite{FGH2013} \label{weakly irreducible}
Let $\mathcal{A}$ be an order $r$ dimension $n$ tensor. If there exists a proper subset $I$ of the set $[n]$ such that
$$a_{i_1i_2\cdots i_r}=0 \quad (\forall \ i_1\in I, \  \mbox {and at least one of the}  \ i_2,\cdots, i_r\notin I).$$
Then $\mathcal{A}$ is called weakly reducible (or sometimes $I$-weakly reducible).
If $\mathcal{A}$ is not weakly reducible, then $\mathcal{A}$ is called weakly irreducible.
\end{definition}

It was proved in \cite{FGH2013} and \cite{YangYang2011} that an $r$-graph $G$
is connected if and only if its adjacency tensor $\mathcal{A(}G\mathcal{)}$
(and so $\mathcal{Q(}G\mathcal{)}$) is weakly irreducible.

\begin{lemma}
\cite{YangYang2011} \label{absolute of a tensor}Let $\mathcal{A}$ and
$\mathcal{B}$ be order $r$ dimension $n$ tensors satisfying $|\mathcal{B}%
|\leq\mathcal{A}$, $\mathcal{A}$ is weakly irreducible. Then

(1). $\rho(\mathcal{B})\leq\rho(\mathcal{A}).$

(2). If $\rho(\mathcal{A})e^{i\theta}$ is an eigenvalue of $\mathcal{B}.$ Then
$\mathcal{B=}$ $e^{i\theta}U^{-(r-1)}\mathcal{A}U$ for some nonsingular
diagonal matrix $U$ all of whose diagonal entries have absolute value 1.
\end{lemma}

Now we can obtain the following result.

\begin{theorem}\label{main result of section 2}
Let $G$ be an $r$-graph. Then $Spec(\mathcal{L}(G))=Spec(\mathcal{Q}(G))$ if and only if $r$ is even and $G$ is odd-colorable.
\end{theorem}

\begin{proof}
First consider the case that $G$ is connected. Then by Theorem \ref{Nikiforov} and Lemma \ref{Shao and Fan} we have
$$Spec(\mathcal{L}(G))=Spec(\mathcal{Q}(G))\Longleftrightarrow Spec(\mathcal{A}(G))=-Spec(\mathcal{A}(G))\Longleftrightarrow G \  \mbox {is odd-colorable and $r$ is even}.$$

Now we consider the case that $G$ is disconnected. Let $G_{1}%
,G_{2},\cdot\cdot\cdot,G_{t}$ be all the connected components of $G$, and the number of vertices of $G_1$ be $n_1$. First we prove the sufficiency part. We have
\begin{align*}
G \  \mbox {is odd-colorable}& \Longrightarrow \mbox {Every} \ G_i \  \mbox {is odd-colorable}  \ (\forall i=1,\cdots,t)\\
&  \Longrightarrow  Spec(\mathcal{L}(G_i))=Spec(\mathcal{Q}(G_i)) \ (\forall i=1,\cdots,t)  \quad \mbox {(by the proof of connected case) }   \\
&  \Longrightarrow  Spec(\mathcal{L}(G))=Spec(\mathcal{Q}(G))  \quad \mbox {(by Lemma \ref{spectrum of union}})
\end{align*}

Now we prove the necessity part of the disconnected case. We will use induction on $t$ (the number of connected components of $G$). Set
$\rho:=\rho(\mathcal{Q(}G\mathcal{)})=\rho(\mathcal{L(}G\mathcal{)}).$ Then $\rho$ is equal to some
$\rho(\mathcal{L(}G_{j}\mathcal{))},$ say $\rho$ $=$ $\rho(\mathcal{L(}%
G_{1}\mathcal{))}.$ Since $|\mathcal{L}(G_1)|=\mathcal{Q}(G_1)$ and $\mathcal{Q}(G_1)$ is nonnegative weakly irreducible, by Lemma \ref{absolute of a tensor} we have $\rho=\rho(\mathcal{Q(}G\mathcal{)})\geq\rho(\mathcal{Q(}G_{1}\mathcal{))\geq}\rho(\mathcal{L(}%
G_{1}\mathcal{))=\rho}$. Thus we also have $\rho$ $=\rho(\mathcal{Q(}G_{1}\mathcal{))}$.
So for the connected $r$-graph $G_{1},$ we have
$\rho(\mathcal{L(}G_{1}\mathcal{))=}\rho(\mathcal{Q(}G_{1}\mathcal{))}.$ Then by Lemma \ref{radius to spectrum}, we have $Spec(\mathcal{L}(G_1))=Spec(\mathcal{Q}(G_1))$, so $G_{1}$ is odd-colorable by the above arguments for the connected case. Now consider the $r$-graph
$G^{\prime}=G_{2}\cup\cdot\cdot\cdot\cup G_{t}.$  Since $G=G_1\cup G^{\prime}$, by Lemma \ref{spectrum of union} we have
$$ Spec(\mathcal{L}(G))=Spec(\mathcal{L}(G_1))^{(r-1)^{n-n_1}}\bigcup  Spec(\mathcal{L}(G^{\prime}))^{(r-1)^{n_1}},$$
$$ Spec(\mathcal{Q}(G))=Spec(\mathcal{Q}(G_1))^{(r-1)^{n-n_1}}\bigcup  Spec(\mathcal{Q}(G^{\prime}))^{(r-1)^{n_1}},$$
Thus $Spec(\mathcal{L}(G))=Spec(\mathcal{Q}(G))$ and $Spec(\mathcal{L}(G_1))=Spec(\mathcal{Q}(G_1))$ imply that $Spec(\mathcal{L}(G^{\prime}))=Spec(\mathcal{Q}(G^{\prime}))$. So by induction on $t$ we obtain that $G^{\prime}$ is also odd-colorable. Therefore we conclude that $G$ is also odd-colorable, since $G=G_1\cup G^{\prime}$ and both $G_1$ and $G^{\prime}$ are odd-colorable.
\end{proof}

As applications of  Theorem \ref{main result of section 2}, we can further obtain the following two results (Theorem \ref{component to graph} and Theorem \ref{characterization of the inverse}).

\begin{theorem}
\label{component to graph}
Let $G$ be an $r$-graph with $r$ even, and $G_1,\cdots, G_k$ be all the connected components of $G$. Then $Spec(\mathcal{L}(G))=Spec(\mathcal{Q}(G))$ if and only if  $Spec(\mathcal{L}(G_i))=Spec(\mathcal{Q}(G_i))$ for every connected component $G_i \ (i=1,\cdots, k)$ of $G$.
\end{theorem}

\begin{proof}
Sufficiency follows from Lemma \ref{spectrum of union}. Now we prove the necessary part. Since  $Spec(\mathcal{L}(G))=Spec(\mathcal{Q}(G))$, we see by Theorem \ref{main result of section 2} that $G$ is odd-colorable. Thus $G_i$ is also odd-colorable, and so by Theorem \ref{main result of section 2} again we have  $Spec(\mathcal{L}(G_i))=Spec(\mathcal{Q}(G_i)) \ (i=1,\cdots, k)$.
\end{proof}

An eigenvalue of a tensor $\mathcal {A}$ is called an H-eigenvalue, if there exists a real eigenvector corresponding to it.

The H-spectrum of a real tensor $\mathcal {A}$, denoted by $Hspec (\mathcal {A})$, is defined to be the set of distinct H-eigenvalues of $\mathcal {A}$. Namely,
$$Hspec (\mathcal {A})=\{\lambda \in \mathbb {R} \ | \ \lambda  \ \mbox {is an H-eigenvalue of $\mathcal {A}.$} \ \}$$

In \cite{ShaoShanWu}(Theorem 2.2), it was proved that when $r$ is even and the $r$-graph $G$ is connected, then
\begin{equation}
Hspec(\mathcal{L}(G))= Hspec (\mathcal{Q}(G)) \Longrightarrow Spec(\mathcal{L}(G))=Spec(\mathcal{Q}(G)) \label{inverse}
\end{equation}
Also the following question was asked in \cite{ShaoShanWu}:

\begin{question}
\label{Hspectra} When $r$ is even, whether the reverse implication of (\ref{inverse}) is true or not?
\end{question}

In \cite{FanKhanTan2016}, Fan et al. showed that the reverse implication of (\ref{inverse}) is not true in the case $r\equiv 0 \ (mod \ 4)$ by taking the generalized power hypergraphs $G^{r,r/2}$ ($G$ is a non-bipartite ordinary graph) as counterexamples.

Now by using Theorem \ref{main result of section 2}, we can show in the following theorem that the reverse implication of (\ref{inverse}) is true in the case  $r\equiv 2 \ (mod \ 4)$, even when $G$ is not connected, thus provide an affirmative answer to Question \ref{Hspectra} for the remaining unsolved case.

\begin{theorem}
\label{characterization of the inverse}
Let $G$ be an $r$-graph with $r\equiv 2 \ (mod \ 4)$, and $Spec(\mathcal{L}(G))=Spec(\mathcal{Q}(G))$. Then we have $Hspec(\mathcal{L}(G))= Hspec (\mathcal{Q}(G)))$.
\end{theorem}

\begin{proof} Let $G_1,\cdots, G_k$ be all the connected components of $G$.
Since $Spec(\mathcal{L}(G))=Spec(\mathcal{Q}(G))$, we obtain by Theorem \ref{main result of section 2} that $G$ is odd-colorable. By  Proposition 12 of \cite{Nikiforov2016} we deduce that $G$ is odd-bipartite since $r\equiv 2 \ (mod \ 4)$. Thus every connected component $G_i$ of $G$ is also odd-bipartite.
Now by  Theorem 2.2 of \cite{ShaoShanWu}, we obtain that $\mathcal{L}(G_i)$ and $\mathcal{Q}(G_i)$ have the same H-spectra for all connected components $G_i \ (i=1,\cdots, k)$ of $G$. Therefore we conclude that $\mathcal{L}(G)$ and $\mathcal{Q}(G)$ have the same H-spectra.
\end{proof}

 Combining Theorem \ref{main result of section 2} and Lemma \ref{radius to spectrum} we have

\begin{corollary}
 \label{spectral radius version}
 Let $G$ be a connected $r$-graph.
Then $\rho(\mathcal{L}(G))=\rho(\mathcal{Q}(G))$ if and only if $r$ is even
and $G$ is odd-colorable.
\end{corollary}

\end{document}